\documentclass[11pt]{amsart}
\usepackage{calc,amssymb,amsthm,amsmath,ulem}
\usepackage{mathtools}
\RequirePackage[dvipsnames,usenames]{xcolor}
\usepackage{hyperref}
\hypersetup{
bookmarks,
bookmarksdepth=3,
bookmarksopen,
bookmarksnumbered,
pdfstartview=FitH,
colorlinks,backref,hyperindex,
linkcolor=Sepia,
anchorcolor=BurntOrange,
citecolor=MidnightBlue,
citecolor=OliveGreen,
filecolor=BlueViolet,
menucolor=Yellow,
urlcolor=OliveGreen
}
\usepackage{alltt}
\usepackage{multicol}
\usepackage{etex}
\usepackage{xspace}
\usepackage{mabliautoref}
\usepackage{colonequals}
\input{kmacros3.sty}
\usepackage{stmaryrd}
\usepackage{tikz}
\usetikzlibrary{matrix,arrows}

\usepackage[left=1.25in,top=1.25in,right=1.25in,bottom=1.25in]{geometry}

\usepackage{verbatim}
\usepackage{enumerate}

\newtheorem*{claim*}{Claim}
\theoremstyle{theorem}

\newtheorem*{theoremA}{Theorem A}
\newtheorem*{theoremB}{Theorem B}
\newtheorem*{theoremC}{Theorem C}
\newtheorem*{theoremD}{Theorem D}
\theoremstyle{remark}
\newtheorem*{*setting}{Setting}

\theoremstyle{remark}
\DeclareMathOperator{\length}{\ell}
\usepackage{graphicx}

\usepackage[all,cmtip]{xy}

\renewcommand{\O}{\mathcal{O}}

\renewcommand{\cf}{{\itshape cf.\xspace\xspace}}

\renewcommand{\:}{\colon}

\DeclareMathOperator{\vol}{vol}

\renewcommand{\phi}{\varphi}

\renewcommand{\theta}{\vartheta}

\renewcommand{\epsilon}{\varepsilon}

\renewcommand{\to}[1][]{\xrightarrow{\ #1\ }}

\renewcommand{\theenumi}{\alph{enumi}}
\begin{document}

\title [$F$-signature of pairs]{\texorpdfstring{$F$}{F}-signature of pairs: Continuity, $p$-fractals and minimal log discrepancies}
\author{Manuel Blickle, Karl Schwede, Kevin Tucker}

\address{ Institut f\"ur Mathematik\\ Johannes Gutenberg-Universit\"at Mainz\\55099 Mainz, Germany}
\email{blicklem@uni-mainz.de}
\address{Department of Mathematics\\ The Pennsylvania State University\\ University Park, PA, 16802, USA}
\email{schwede@math.psu.edu}
\address{Department of Mathematics\\ Princeton University\\ Princeton, NJ, 08544, USA}
\email{kftucker@math.princeton.edu}

\thanks{The first author was partially supported by a Heisenberg Fellowship and the SFB/TRR45}
\thanks{The second author was partially supported by the NSF grant DMS \#1064485}
\thanks{The third author was partially supported by the NSF postdoctoral fellowship DMS \#1004344}

\maketitle

\begin{abstract} This paper contains a number of observations on the \mbox{$F$-signature} of triples $(R,\Delta,\ba^t)$ introduced in our previous joint work \cite{BlickleSchwedeTuckerFSigPairs1}. We first show that the \mbox{$F$-signature} $s(R,\Delta,\ba^t)$ is continuous as a function of $t$, and for principal ideals $\ba$ even convex.  We then further deduce, for fixed $t$, that
the $F$-signature is lower semi-continuous as a function on $\Spec R$ when $R$ is regular and $\ba$ is principal. We also point out the close relationship of the signature function in this setting to the works of Monsky and Teixeira  on Hilbert-Kunz multiplicity and $p$-fractals \cite{MonskyTeixeiraPFractals1,MonskyTeixeiraPFractals2}.  Finally, we conclude by showing that the minimal log discrepancy of an arbitrary triple $(R,\Delta,\ba^t)$ is an upper bound for the $F$-signature.
\end{abstract}


\section{Introduction}
The $F$-signature $s(R)$ of a reduced local ring $(R,\bm,k)$ in positive characteristic gives a measure of singularities by analyzing the asymptotic behavior of the number of splittings ($F$-splittings) of large iterates of the Frobenius endomorphism.
It is a real number between $0$ and $1$, $s(R)=1$ characterizing regular rings \cite{HunekeLeuschkeTwoTheoremsAboutMaximal,YaoObservationsAboutTheFSignature}, and $s(R) > 0$ characterizing $F$-regular rings \cite{AberbachLeuschke}.
More generally, this rather subtle invariant is thought to encode various arithmetic and geometric properties of $R$; for example, the $F$-signature recovers the group order of tame finite quotient singularities (\cite[Example 18]{HunekeLeuschkeTwoTheoremsAboutMaximal},\cite[Remark 4.7]{YaoObservationsAboutTheFSignature},\cite[Corollary 4.13]{TuckerFSignatureExists}), and is closely related to the theory of Hilbert-Kunz multiplicity.

In our previous work \cite{BlickleSchwedeTuckerFSigPairs1}, we extended the notion of $F$-signature to the setting commonly considered in birational algebraic geometry.  Specifically, we consider triples $(R,\Delta,\ba^t)$ where $R$ is a local normal domain, $\Delta$ is an effective \mbox{$\mathbb{R}$-divisor} on $X = \Spec R$, and $\ba \subseteq R$ is a non-zero ideal with coefficient $t \in \mathbb{R}_{\geq 0}$ (\cf~\cite{BlickleSchwedeTakagiZhang}).  Once more, the $F$-signature of such a triple is an asymptotic measure of a number of $F$-splittings, however we restrict the set of ``allowable'' splittings by taking into account the additional data of the triple. Importantly, as before in the absolute case above, the positivity of the $F$-signature characterizes strongly $F$-regular triples \cite[Theorem 3.18]{BlickleSchwedeTuckerFSigPairs1}.

The purpose of the present paper is to point out a number of interesting properties of the $F$-signature of triples and relate it to previously studied invariants of singularities. We spend most of our time studying the function
\[
t \mapsto s(R, \Delta, \ba^t).
\]
Most notably, we first show that this function is Lipschitz and hence, in particular, continuous as a function of $t$. In case $R$ is regular and $\ba = \langle f \rangle$ is principal, substantially more can be said: the \mbox{$F$-signature} $s(R, f^{t})$ is shown to be convex, and its derivatives recover both the Hilbert-Kunz multiplicity and \mbox{$F$-signature} of $R/ \langle f \rangle$.  Furthermore, it is closely related to the works of Monsky and Teixeira on Hilbert-Kunz multiplicity and $p$-fractals \cite{MonskyTeixeiraPFractals1,MonskyTeixeiraPFractals2}. As an application of these results we show that, for \emph{fixed} $t$, the $F$-signature $s(R, f^{t})$ is lower semi-continuous as a function on $\Spec R$.  Finally, we conclude by showing that the minimal log discrepancy of an arbitrary triple $(R,\Delta,\ba^t)$ is an upper bound for the $F$-signature.

To simplify notation when summarizing our results more precisely in this introduction, we consider only the case of pairs $(R,f^t)$ with $R$ a $d$-dimensional $F$-finite normal local domain with \emph{perfect} residue field, $f\in R$ a non-zero element, and $t \in \mathbb{R}_{\geq 0}$.
For each $e > 0$, we consider the maximal rank $a_{e}$ of an $R$-free direct summand of $F^e_*R \cong R^{1/p^e}$.  In other words, $a_{e}$ is determined by writing
\[
F^e_* R \cong R^{\oplus a_e} \oplus M_e
\]
as $R$-modules where $M_{e}$ has no free direct summand, and
the $F$-signature is defined to be the limit \cite{SmithVanDenBerghSimplicityOfDiff,HunekeLeuschkeTwoTheoremsAboutMaximal}
\[
    s(R)= \lim_{e \to \infty} {a_e \over p^{ed} } .
\]
Recently, the third author has shown that this limit exists in general \cite{TuckerFSignatureExists}. The \mbox{$F$-signature} of the pair $(R,f^t)$ is defined similarly after restricting which direct summands are taken into account; setting $\cramped{a_{e}^{f^t}}$ to be the maximal rank of an $R$-free direct summand of $F^e_* R \simeq R^{\oplus a_{e}^{f^{t}}} \oplus M_{e}^{f^{t}}$ where the associated projections $F^{e}_{*}R \to R$ factor through multiplication by $F^e_* f^{\lceil t (p^e - 1) \rceil}$ on $F^{e}_{*}R$, the $F$-signature of the pair $(R,f^t)$ is the limit
\[
    s(R, f^t)= \lim_{e \to \infty} {a_{e}^{f^t} \over p^{e} }
\]
shown to exist in \cite{BlickleSchwedeTuckerFSigPairs1}.

The first new result obtained herein is the continuity of $F$-signature in terms of the parameter $t$.
\begin{theoremA} [\autoref{thm.Continuity}]
Suppose that $R$ is an $F$-finite local ring and $f \in R$.  Then the function
\[
t \mapsto s(R, f^t)
\]
is a continuous function of $t$ on $[0,\infty)$.
\end{theoremA}
\noindent This result also holds for triples $(R, \Delta, \ba^t)$ as well.  Using this continuity result (and the method of its proof), we also obtain the stronger statement
\begin{theoremB} [\autoref{thm.Convexity}]
Suppose that $R$ is an $F$-finite local ring and $f \in R$.  Then the function
\[
t \mapsto s(R, f^t)
\]
is convex on $[0, \infty)$.
\end{theoremB}
\noindent
However, in contrast to the continuity statement, we are only able to show that the convexity generalizes to triples where the ideal $\ba = \langle f \rangle$ is principal.

If $R$ is regular and $t = a/p^c$ for $a,c \in \bN$, then $s(R, f^t)$ can be computed from a single length as opposed to an abstract limit; see \autoref{prop.EasyComputationOfFSignatureForRegular}.  Interestingly, this length has also featured prominently in the works of Monsky and Teixeira on Hilbert-Kunz multiplicity and $p$-fractals (see \cite{MonskyTeixeiraPFractals1,MonskyTeixeiraPFractals2}). In \autoref{sec.RelationWithPFractals}, we explain this connection in detail and further clarify the relationship between the $F$-signature function $s(R,f^{t})$ and the invariants of the hypersurface $R / \langle f \rangle$ (see \autoref{thm.identifyderivatives}).  Indeed, we have that the right derivative of $s(R, f^{t})$ exists at $t = 0$ and equals the negative of the Hilbert-Kunz multiplicity of $R/\langle f \rangle$; similarly, the left derivative of $s(R, f^{t})$ exists at $t = 1$ and equals the negative of the \mbox{$F$-signature} of $R/\langle f \rangle$.  In particular, it follows that $R/ \langle f \rangle$ is strongly \mbox{$F$-regular} if and only if $s(R,f^{t})$ is \emph{not} differentiable at $t = 1$.
Finally, we also numerically compute the example of the \mbox{$F$-signature} function $t \mapsto s(R, f^t)$ of the cusp $f := y^2 + x^3  \in k[x,y]_{\langle x,y \rangle} =: R$ in some small characteristics.

In the absolute setting, a difficult and important open problem is to show that the \mbox{$F$-signature} is lower semi-continuous as a function on the prime spectrum of a ring \cite{EnescuYaoLowerSemicontinuity}. However, using the computation of the $F$-signature for certain pairs as a single length together with the continuity in the scaling parameter, we are able to show the following result.
\begin{theoremC} [\autoref{thm.SemiContinuityPair}]
If $R$ is regular and not necessarily local with $0 \neq f \in R$ and fixed $t \geq 0$, then the function $\Spec R \to \bR$ defined by the rule
\[
\bq \mapsto s(R_{\bq}, f^t)
\]
is lower semi-continuous in the Zariski topology of $\Spec R$.
\end{theoremC}

The $F$-signature is only interesting (\textit{i.e.} non-zero) for strongly $F$-regular triples, a notion which corresponds to the characteristic zero property of Kawamata log terminal singularities via reduction to prime characteristic $p > 0$. The basic invariant used to study Kawamata log terminal singularities are the minimal log discrepancies, which we are able to compare to the $F$-signature in the final section of this article.
\begin{theoremD} [\autoref{thm.F-sigLeqMLD}]
Suppose that $R$ is normal $\bQ$-Gorenstein and $F$-finite. Then for all $x \in X = \Spec R$ we have $s(\O_{X,x}, f^t) \leq {\mld}(x; X, f^t)$.
\end{theoremD}
\noindent
The above result is both novel and interesting in the absolute case, though it once more generalizes to triples $(R, \Delta, \ba^t)$ as well; see \autoref{thm.F-sigLeqMLD}. We are also able to include several further improvements; in particular, see \autoref{cor.MldImprovement}.
\vskip 9pt
\noindent{Acknowledgements:}  The authors would like to thank Craig Huneke and Karen Smith for inspiring conversations.  The authors would also like to thank the all the referees and Mircea Musta{\c{t}}{\u{a}} for numerous useful comments on previous drafts.  The authors worked on this paper while visiting the Johannes Gutenberg-Universit\"at Mainz during the summers of 2010 and 2011.  These visits were funded by the SFB/TRR45 \emph{Periods, moduli, and the arithmetic of algebraic varieties}.

\section{The \texorpdfstring{$F$}{F}-signature for triples}
\label{sec.Defn}
We begin by fixing notation and recalling the definition of and formulae for the \mbox{$F$-signature} of a triple; in particular, the description found in \autoref{thm.ConvenientDefFsigTrip}  will suffice for a majority of this article.
Throughout this paper $R$ is a $d$-dimensional \mbox{$F$-finite} local normal domain of characteristic $p > 0$ with maximal ideal $\bm$ and residue field $k$. We are concerned with triples $(R,\Delta,\ba^t)$ where $\Delta$ is an effective \mbox{$\bR$-divisor} on $\Spec R$, and $\ba \subseteq R$ is a non-zero ideal with parameter $t \in \bR_{\geq 0}$; if $\Delta = 0$ or $\ba = R$, we will omit these terms from the notation. For any $R$-module $M$, we use $F^e_* M$ to denote the $R$-module which agrees with $M$ as an Abelian group but where $R$-multiplication is twisted by the $e$-iterated Frobenius $F^e \colon R \to R$. In other words, the $R$-module structure on $F^e_*M$ is given by $r\cdot m = r^{p^e} m$ for $r \in R$ and $m \in M$. The $F$-finiteness condition means precisely that $F^e_* R$ is a finitely generated $R$-module for every $e > 0$.

We define a \emph{$p^{-e}$-linear map} $\varphi \colon M \to M$ to be an $R$-linear map $\varphi : F^e_* M \to M$, \textit{i.e.} $\varphi$ is an additive endomorphism of $M$ satisfying $\varphi(r^{p^e}m)=r\varphi(m)$ for all $m \in M$ and $r \in R$. The prototypical example of a $p^{-e}$-linear map is a Frobenius splitting, or better still an associated projection map onto an $R$-module direct summand of $F^e_*R$ isomorphic to $R$.  The $F$-signature of $R$ is an asymptotic measure of the number of all such summands as $e$ increases.

More generally, to define the $F$-signature of a triple $(R,\Delta, \ba^{t})$ we limit the direct summands enumerated by taking into account the additional data of the triple. The effective $\bR$-divisor $\Delta$ yields an inclusion
\[
    R \subseteq R(\lceil (p^e-1) \Delta \rceil) \colonequals  \{ f  \,|\, f \in \text{Frac}(R)
    \text{ with }\operatorname{div}(f) +\lceil (p^e-1)\Delta \rceil \geq 0 \}
\]
and induces (by restriction) an inclusion $\Hom_R(F^e_*R(\lceil (p^e-1) \Delta \rceil), R) \subseteq \Hom_R(F^e_*R,R)$. A free direct summand $R^{\oplus n} \subseteq F^e_*R$ is called an \emph{$(R,\Delta)$-summand} if each of the associated projection maps $\varphi \in \Hom_R(F^e_*R,R)$ belongs to $\Hom_R(F^e_*R(\lceil (p^e-1) \Delta \rceil), R)$. If additionally, all the  associated projection maps $\varphi \in \Hom_R(F^e_*R,R)$ lie in the submodule \mbox{$\Hom_R(F^e_*R(\lceil (p^e-1) \Delta \rceil), R) \cdot F^e_*\ba^{\lceil t(p^e-1) \rceil}$} then it is called a $(R,\Delta,\ba^t)$-summand\footnote{In other words,  each $\varphi$ can be written as a sum $\sum \psi _i( F^e_*a_i \cdot \underline{\phantom{m}})$ for some $a_i \in \ba^{\lceil t(p^e-1) \rceil}$ and $\psi_i \in \Hom_R(F^e_*R(\lceil (p^e-1) \Delta \rceil), R)$.}. Thus, we are led to the definition of the $F$-signature for triples.
\begin{definition}\cite{BlickleSchwedeTuckerFSigPairs1}
The \emph{$e$-th $F$-splitting number of a triple $(R, \Delta, \ba^t)$} is the maximal rank $a_{e}^{\Delta, \ba^t}$ of a free $(R,\Delta,\ba^t)$-summand  of $F^{e}_{*}R$ as an $R$-module. If $d=\dim R$ and $\alpha(R)=\log_p [k:k^p]$, the limit
\[
s(R,\Delta, \ba^t) = \lim_{e \to \infty}
\frac{a_{e}^{\Delta, \ba^t}}{p^{e(d+\alpha(R))}}
\]
exists and is called the \emph{$F$-signature} of $(R, \Delta, \ba^t)$.
\end{definition}

A key result about the $F$-signature states that $s(R,\Delta,\ba^t)$ is positive if and only if the triple $(R,\Delta,\ba^t)$ is strongly $F$-regular \cite[Theorem 3.18]{BlickleSchwedeTuckerFSigPairs1} (\cf~\cite{AberbachLeuschke}). In this case, $R$ itself is strongly $F$-regular and hence normal and Cohen-Macaulay.

Prior to the last section, we shall rely solely on the following explicit description of the $F$-signature for triples obtained in \cite{BlickleSchwedeTuckerFSigPairs1}. As in \textit{loc. cit.} Definition 3.3, set
\[
    I^\Delta_e \colonequals \{ r \in R\, |\, \phi(r) \in \bm \text{ for all } \phi \in \Hom_R(F^e_*R(\lceil (p^e-1)\Delta \rceil),R) \}.
\]
\begin{proposition}\label{thm.ConvenientDefFsigTrip}
Let $(R,\Delta,\ba^t)$ be a triple.  According to \cite[Proposition 3.5]{BlickleSchwedeTuckerFSigPairs1}, we have
$a_e^{\Delta,\ba^t} =p^{e\alpha(R)}\length(R/(I^\Delta_e:\ba^{\lceil t(p^e-1) \rceil}))$ and hence
\begin{align*}
s(R,\Delta,\ba^t) &= \lim_{e \to \infty} \frac{1}{p^{ed}}\length(R/(I^\Delta_e:\ba^{\lceil t(p^e-1) \rceil})). \\
\intertext{Additionally, using \cite[Proposition 4.17]{BlickleSchwedeTuckerFSigPairs1}, we also have}
s(R,\Delta,\ba^t) &=\lim_{e \to \infty} \frac{1}{p^{ed}}\length(R/(I^\Delta_e:\ba^{\lceil tp^e \rceil})).
\end{align*}
\end{proposition}

\begin{remark}
The definition of the $F$-signature is further generalized in \cite{BlickleSchwedeTuckerFSigPairs1} using the notion of a Cartier subalgebra on $R$. Let us briefly recall this often convenient point of view: the \emph{total Cartier algebra on $R$} is the non-commutative graded ring
\[
    \sC^R = \bigoplus_{e=0}^\infty \sC^R_e
\]
where $\sC^R_e = \Hom_R(F^e_*R,R)$ and multiplication is given by composition of additive maps. A \emph{Cartier subalgebra} is a graded subring $\sD$ of $\sC^R$ with $\sD_0=R$. One says that a free direct summand of $F^e_*R$ is a $(R,\sD)$-summand if all associated projection maps $F^e_*R \to R$ belong to $\sD_e$. The $F$-signature of $(R,\sD)$ is then roughly\footnote{For a general Cartier subalgebra, one has to be slightly more careful when taking the limit in that some $\sD_e$ might be zero. This pathology does not occur in the case that $\sD$ comes from a triple.} given by
\[
 s(R,\sD) = \lim_{e \to \infty} {a_e^\sD \over p^{e(d+\alpha(R))}}
\]
where $a_e^\sD$ is the maximal rank of a $(R,\sD)$-free summand of $F^e_*R$. With this setup it is easy to verify that both $\sC^{\Delta} = \bigoplus \sC^{\Delta}_{e}$ with $\sC^{\Delta}_e \colonequals \Hom_R(F^e_*R(\lceil (p^e-1) \Delta \rceil),R)$ and $\sC^{\Delta, \ba^{t}} = \oplus \sC^{\Delta, \ba^{t}}_{e}$ with $\sC^{\Delta,\ba^t}_e \colonequals \sC^\Delta_e \cdot F^e_* \ba^{\lceil t(p^e-1)\rceil}$ form Cartier subalgebras. See \cite{BlickleSchwedeTuckerFSigPairs1,SchwedeTuckerTestIdealSurvey} for details.
\end{remark}

\section{Continuity and convexity of \texorpdfstring{$F$}{F}-signature}
\label{sec.Continuity}

An important basic case  of $F$-signature of pairs is that of a simple normal crossing divisor, \textit{i.e.} where $R = k\llbracket x_1, \dots, x_n\rrbracket$ and $\Delta = t_1 \Div_X(x_1) + \ldots + t_n \Div_X(x_n)$.  As worked out in \cite[Example 4.19]{BlickleSchwedeTuckerFSigPairs1}, we have
\begin{equation}
\label{eq.SignatureOfSNC}
s(R, \Delta) = (1-t_1)(1-t_2) \dots (1-t_n)
\end{equation}
so long as $t_i \leq 1$ for all $i$ (otherwise the signature is zero). In fact, this is but a special case of the formula for the $F$-signature of monomial ideals.

\begin{theorem}\cite[Theorem 4.20]{BlickleSchwedeTuckerFSigPairs1}
Let $R = k \llbracket x_{1}, \ldots, x_{n} \rrbracket$ where $k$ is an \mbox{$F$-finite} field of characteristic $p >0$.
Suppose $\ba \subseteq R$ is a monomial ideal with Newton polyhedron $\bP_{\ba} \subseteq \R^{ n}$ and $t \in \R_{> 0}$.  Then the $F$-signature of $(R, \ba^{t})$
\[
s(R, \ba^{t}) = \vol_{\R^{ n}}(t\bP_{\ba} \cap [0,1]^{ n})
\]
equals the Euclidean volume of the intersection of $t \bP_{\ba}$ with the unit cube $[0,1]^{ n} \subseteq \R^{n}$.  
\end{theorem}

This result has recently been generalized to toric varieties by M.~Von Korff in \cite{VonKorffFSigOfAffineToric} (\cf~ \cite{WatanabeYoshidaMinimalRelativeHKMult,SinghFSignatureOfAffineSemigroup}). What one notices in any of these basic cases is that the \emph{signature function}
\[
t \mapsto s(R, \ba^t)
\]
is a continuous function of $t$. In this section, we verify this property in general, even after additionally incorporating a divisor $\Delta$. We additionally prove that in the case that $\ba$ is principal, that the signature function is convex.

\begin{theorem}\label{thm.Continuity}
  Suppose $R$ is a $d$-dimensional $F$-finite local normal domain with dimension $d$, $\Delta$ is an effective $\R$-divisor on $\Spec(R)$, and $\ba \subseteq R$ is a non-zero ideal.  Then the function $t \mapsto s(R,\Delta,\ba^{t})$ is continuous.
\end{theorem}

\begin{proof}
  Without loss of generality, we may assume $R$ is strongly $F$-regular (else the function in question is identically zero) and hence Cohen-Macaulay.  Choose a non-zero element $x_{1} \in \ba$, and complete it to a system of parameters $x_{1}, x_{2}, \ldots, x_{d}$ for $R$.  In particular, $x_{1}, x_{2}, \ldots, x_{d}$ form a regular sequence, so that
\begin{equation}
\label{eq:1}
\length_{R}\left(R/ \langle x_{1}^{n_{1}}, x_{2}^{n_{2}},\ldots, x_{d}^{n_{d}} \rangle \right) = n_{1} n_{2} \cdots n_{d} \length_{R}\left(R/\langle x_{1}, x_{2}, \ldots, x_{d} \rangle \right)
\end{equation}
for any $n_{1}, n_{2}, \ldots, n_{d} \in \ZZ_{>0}$.

Using the description of the $F$-signature recalled above in \autoref{thm.ConvenientDefFsigTrip} we have that
\[
s(R,\Delta, \ba^{t}) = \lim_{e \to \infty} \frac{1}{p^{ed}} \length_{R}\left( R / (I_{e}^{\Delta} : \ba^{\lceil tp^{e} \rceil} )\right) .
\]
Fix $e_{0}\in \ZZ_{>0}$; for $e \geq e_{0}$, we have that
\begin{equation}
\label{eqn.DifferenceEqualsColon}
\begin{array}{rl}
0 \leq & \displaystyle \length_{R}\left( R / (I_{e}^{\Delta} : \ba^{\lceil tp^{e} \rceil} )\right)  - \length_{R}\left( R / (I_{e}^{\Delta} : \ba^{\lceil (t+\frac{1}{p^{e_{0}}})p^{e} \rceil} )\right) \medskip \\
 = & \displaystyle \length_{R}\left( \frac{\left( I_{e}^{\Delta} : \ba^{\lceil (t+\frac{1}{p^{e_{0}}})p^{e} \rceil}  \right)}{(I_{e}^{\Delta} : \ba^{\lceil tp^{e} \rceil} )}\right) \; \;
 = \; \; \displaystyle \length_{R}\left( \frac{\left( (I_{e}^{\Delta} : \ba^{\lceil tp^{e} \rceil} ): \ba^{p^{e-e_{0}}} \right)}{(I_{e}^{\Delta} : \ba^{\lceil tp^{e} \rceil} )}\right).
\end{array}
\end{equation}
Now, since $x_{1} \in \ba$, we have
\begin{equation}
\label{eq:smalldiff}
\begin{array}{rl}
\length_{R}\left( \frac{\left( (I_{e}^{\Delta} : \ba^{\lceil tp^{e} \rceil} ): \ba^{p^{e-e_{0}}} \right)}{(I_{e}^{\Delta} : \ba^{\lceil tp^{e} \rceil} )}\right) &\leq \length_{R}\left( \frac{\left( (I_{e}^{\Delta} : \ba^{\lceil tp^{e} \rceil} ): x_{1}^{p^{e-e_{0}}} \right)}{(I_{e}^{\Delta} : \ba^{\lceil tp^{e} \rceil} )}\right) \vspace{6pt}\\
& = \length_{R} \left( \frac{R}{\langle x_{1}^{p^{e-e_{0}}},  (I_{e}^{\Delta} : \ba^{\lceil tp^{e} \rceil} )\rangle}\right) \vspace{6pt}\\
& \leq \quad \frac{p^{ed}}{p^{e_{0}}}\length_{R}\left( \frac{R}{\langle x_{1}, \ldots, x_{d} \rangle}\right).
\end{array}
\end{equation}
where the equality in the second line comes from the fact that the lengths of the kernel and cokernel of multiplication by $x_{1}^{p^{e-e_{0}}}$ on $R/ (I_{e}^{\Delta} : \ba^{\lceil tp^{e} \rceil} )$ must be equal, and the inequality in the third is just by the inclusion
\[
x_{2}^{p^{e}}, \ldots, x_{d}^{p^{e}} \in \bm^{[p^{e}]} \subseteq I_{e}^{\Delta} \subseteq (I_{e}^{\Delta} : \ba^{\lceil tp^{e} \rceil} )
\]
and \autoref{eq:1}.  Thus, dividing by $p^{ed}$ and letting $e \to \infty$ gives
\[
0 \leq s(R, \Delta, \ba^{t}) - s(R, \Delta, \ba^{t + \frac{1}{p^{e_{0}}}}) \leq \frac{1}{p^{e_{0}}}\length_{R}\left( \frac{R}{\langle x_{1}, \ldots, x_{d} \rangle}\right).
\]
Using that $s(R,\Delta,\ba^{t})$ is non-increasing in $t$, continuity now follows immediately.
\end{proof}

The proof also shows that:

\begin{corollary}
With the notation above, $s(R,\Delta,\ba^{t})$ is a Lipschitz function in $t$ with constant $\length_{R}(R/\langle x_{1}, \ldots, x_{d} \rangle)$.
\end{corollary}

\begin{remark}
  Furthermore, if $\sD$ is a Cartier subalgebra on $R$, essentially the same proof shows that the function $s(R,\sD,\ba^{t})$ is continuous in $t$ (replacing $I_{e}^{\Delta}$ with $I_{e}^{\sD}$ throughout, and using only those $e$ which appear in the semigroup of $\sD$, see \cite{BlickleSchwedeTuckerFSigPairs1}).
\end{remark}

A close inspection of the proof of the continuity in \autoref{thm.Continuity} shows that in the case that $\ba = \langle f \rangle$ is principal one can do better: The function $t \mapsto s(R, \Delta, f^t)$ is convex -- which implies that the signature function is differentiable at all but countably many points.

\begin{theorem}\label{thm.Convexity}
Suppose that $R$ is an $F$-finite normal local domain, $\Delta$ is an effective \mbox{$\bR$-divisor} on $\Spec(R)$,  and $f \in R$ is a non-zero element. Then the function $t \mapsto s(R,\Delta,f^{t})$ is convex on the interval $[0, \infty)$.
\end{theorem}

\begin{proof}
If $0 \leq t < t_{1} < t_{2}$, we need to verify
\[
D(t, t_{1}) = \frac{s(R,\Delta,f^{t_{1}}) - s(R,\Delta,f^{t})}{t_{1} - t} \leq \frac{s(R,\Delta,f^{t_{2}}) - s(R,\Delta,f^{t_1})}{t_{2} - t_1} = D(t_1, t_{2}).
\]
By continuity, it suffices to check this inequality on the dense set of rational numbers whose denominators are powers of $p$.  Thus, for some $c  \in \bZ_{>0}$, we will assume that
\[
t = \frac{a}{p^{c}} \qquad t_{1} = \frac{a_{1}}{p^{c}} \qquad t_2=\frac{a_{2}}{p^{c}}
\]
where $a < a_{1} < a_{2} \in \Z_{>0}$. In fact, it is enough to treat the case $a_1=a+1, a_2=a+2$.

Following the proof of \autoref{thm.Continuity} we notice that the first inequality in Equation \autoref{eq:smalldiff} is an equality whenever $\ba=\langle f \rangle$ is principal. Combining this with Equation \autoref{eqn.DifferenceEqualsColon} we obtain
\[
\length_{R}\left( R / (I_{e}^{\Delta} : f^{\lceil (\frac{a}{p^{c}})p^{e} \rceil} )\right)  - \length_{R}\left( R / (I_{e}^{\Delta} : f^{\lceil (\frac{a+1}{p^{c}})p^{e} \rceil} )\right)\vspace{6pt}
 =  \length_{R}\left(  \frac{R}{\langle (I_{e}^{\Delta}:f^{ap^{e-c}}),f^{p^{e-c}}\rangle}\right).
\]
Dividing through by $p^{ed}$ and taking limits as $e \to \infty$ then gives
\[
s(R,\Delta, f^{\frac{a}{p^{c}}}) - s(R,\Delta,f^{\frac{a+1}{p^{c}}}) = \lim_{e \to \infty}\frac{1}{p^{ed}}\length_{R}\left(  \frac{R}{\langle (I_{e}^{\Delta}:f^{ap^{e-c}}),f^{p^{e-c}}\rangle}\right).
\]
Since for $a<b$ we have $\langle f^{ap^{e-c}} \rangle \supseteq \langle f^{bp^{e-c}} \rangle$ and hence
\[
\langle (I_{e}^{\Delta}:f^{ap^{e-c}}),f^{p^{e-c}}\rangle \subseteq \langle (I_{e}^{\Delta}:f^{bp^{e-c}}),f^{p^{e-c}}\rangle.
\]
It follows that
\[
D\left(\frac{a}{p^{c}},\frac{a+1}{p^{c}}\right) \leq D\left(\frac{a+1}{p^{c}},\frac{a+2}{p^{c}}\right)
\]
for all $a \geq 0$ as desired.
\end{proof}

From convexity, we easily obtain the following corollary.

\begin{corollary}
\label{cor.DerivativesAt0AreLipschitzConstants}
Suppose that $R$ is an $F$-finite normal local domain, $\Delta$ is an effective \mbox{$\bR$-divisor} on $\Spec(R)$,  and $f \in R$ is a non-zero element. Then the left $D_{-}s(R,\Delta, f^{t})$ and right derivatives $D_{+}s(R,\Delta,f^{t})$ exist on all of $(0,\infty)$ and are non-decreasing.  Furthermore, the right derivative $D_{+}s(R,\Delta,f^{0})$ exists, and $-D_{+}s(R,\Delta,f^{0})$ is the sharp (in other words, best) Lipschitz constant for $s(R,\Delta,f^{t})$.
\end{corollary}

\begin{proof}
 Since $s(R,\Delta,f^{t})$ is Lipschitz and convex on $[0,\infty)$, it can be extended to a convex function on all of $\bR$.  This immediately implies the desired existence of the right and left derivatives at all points, including the existence of $D_{+}s(R,\Delta,f^{0})$.  The fact that they are non-decreasing is an immediate consequence of convexity.  Therefore, since $s(R,\Delta,f^{t}) = 0$ when $t \geq 1$, it follows that both left and right derivatives are non-positive at every point and
\[
\sup_{0 \leq t_{1} < t_{2}} \left| \frac{s(R,\Delta,f^{t_{2}}) - s(R,\Delta,f^{t_{1}})}{t_{2} - t_{1}}\right| = -D_{+}s(R,\Delta,f^{0})
\]
as claimed.
\end{proof}

\section{Relation with \texorpdfstring{$p$}{p}-fractals}
\label{sec.RelationWithPFractals}

In this section, we specialize to the case where $(R,\bm,k)$ is a \emph{regular} local $F$-finite ring of dimension $d$ and $\Delta=0$, where it is elementary to see $I^{\Delta = 0}_e=\bm^{[p^e]}$ (\cf~\cite[Theorem 4.13]{BlickleSchwedeTuckerFSigPairs1}).  Hence, by \autoref{thm.ConvenientDefFsigTrip}, the $F$-signature of $(R,\ba^{t})$ is given as the limit
\[
    s(R,\Delta,\ba^t) =\lim_{e \to \infty} \frac{1}{p^{ed}}\length_{R}(R/(\bm^{[p^e]}:\ba^{\lceil tp^e \rceil})).
\]
However, when considering a pair $(R, f^{t})$ where $\ba = \langle f \rangle$ is principal and $t = a / p^{c}$ is a rational number whose denominator is a power of $p$, we may simplify even further.
In this case, the above limit becomes
\[
\lim_{e \to \infty} {1 \over p^{ed}}\ell_R(R/(\bm^{[p^e]} : f^{ ap^{e-c}})).
\]
But since $R$ is regular, $(\bm^{[p^e]} : f^{ ap^{e-c}}) = (\bm^{[p^c]} : f^{a})^{[p^{e-c}]}$ and the limit is just a scalar multiple of the Hilbert-Kunz multiplicity $e_{\textnormal{HK}}((\bm^{[p^c]} : f^{a}), R)$ \cf \cite{MonskyHKFunction}.  Again by regularity of $R$ we then have $e_{\textnormal{HK}}((\bm^{[p^c]} : f^{a}), R) = \ell_R(R/(\bm^{[p^c]} : f^{a}))$. Summarizing, this yields:
\begin{proposition}
\label{prop.EasyComputationOfFSignatureForRegular}
Suppose that $(R, \bm, k)$ is an $F$-finite regular $d$-dimensional domain and $f \in R$ is a non-zero element.  Then for any integers $a, c > 0$
\[
s(R, f^{a/p^c}) = {1 \over p^{cd}}  \ell_R(R/(\bm^{[p^c]} : f^{a})).
\]
\end{proposition}

The just described function appeared previously in the work of P.~Monsky and P.~Teixeira on $p$-fractals, \cite{MonskyTeixeiraPFractals1,MonskyTeixeiraPFractals2}. To show the relationship of $F$-signature with their work, we fix $k$ to be a finite (hence perfect) field and let $R = k \llbracket x_1, \dots, x_n \rrbracket$, $\bm = \langle x_1, \dots, x_n \rangle$, and $0 \neq f \in R$. Let us denote by
\[
\sI := \{ x \in \bQ\, |\, x \in [0,1], x = a/p^c, a \in \bN, c \in \bN \}
\]
the set of rational numbers in $[0,1]$ whose denominators are powers of $p$.  The authors of \cite{MonskyTeixeiraPFractals2} extensively study the function
\[
\phi_f \colon \sI \to \bQ
\]
defined by
\[
\phi_f(a/p^{c}) = p^{-cd} \cdot \dim_k \left(R/(\bm^{[p^c]} + \langle f^a \rangle)\right).
\]
In particular, the authors studied the implications of whether this function is a \mbox{$p$-fractal}, see \cite[Section 2]{MonskyTeixeiraPFractals1}.  By definition, a function $\varphi : \sI \to \bQ$ is a \emph{$p$-fractal} if the $\bQ$-vector-subspace of the set of functions $\sF = \{ \psi : \sI \to \bQ \}$ spanned by the functions
\[
    \{ \phi({\textstyle\frac{t+b}{p^e}}) | e \geq 0, 0 \leq b \leq p^e \}
\]
is finite dimensional as a $\bQ$-vector space.  Indeed, in \cite{TeixeiraThesis,MonskyTeixeiraPFractals2} it was shown that if $\phi_f$ is a $p$-fractal, then the Hilbert-Kunz series\footnote{The Hilbert-Kunz series of $f \in k\llbracket x_1, \dots, x_s \rrbracket =: R$ is the formal power series in $\bZ\llbracket z \rrbracket$:
\[
\sum_{n = 0}^{\infty} \dim_k \big(R/\langle x_1^{p^n}, \dots, x_s^{p^n}, f \rangle \big) z^n.
\]}
of $f \in R$ is a rational function.

Since we have the short exact sequence
\begin{equation}
\label{eq:standardtrick}
0 \to R/(\bm^{[p^e]} : f^a) \to[\cdot f^{a}] R/\bm^{[p^e]} \to R/(\bm^{[p^e]} + \langle f^a \rangle) \to 0
\end{equation}
we see immediately that
\begin{equation}
\label{eq:pfractalrelationship}
\varphi_f(a/p^e) := 1 - s(R, f^{a/p^e}),
\end{equation}
and in particular we have the following result.

\begin{theorem}
\cite[Theorem 1]{MonskyTeixeiraPFractals1}
  Consider $R = k \llbracket x,y \rrbracket$ where $k$ is a finite field of characteristic $p$ and suppose $f \in R$ is a non-zero element.  Then the function $\frac{a}{p^{c}} \mapsto s(R,f^{\frac{a}{p^{c}}})$ is a $p$-fractal.
\end{theorem}

\begin{remark}
In arbitrary dimension, the $F$-signature $t \mapsto s(R,f^{t})$ thus gives another interpretation via \eqref{eq:pfractalrelationship} of the functions $\varphi_f$ studied by Monsky and Teixeira.  This is particularly novel in that it is geometric in nature: setting $X = \Spec(R)$ and $D = \divisor_{X}(f)$, the function $\varphi_{f}$ is measuring the singularities of the pairs $(X, tD)$ given by scaling the divisor $D$.  Similar techniques are commonplace throughout (complex) birational algebraic geometry (see also Section~\ref{sec:minim-log-discr}).
\end{remark}

The primary motivation of Monsky and Teixeira for studying the functions $\varphi_{f}$ was in order to compute the Hilbert-Kunz multiplicity of the quotient $R / \langle f \rangle$.
The following result further clarifies the relationship of these functions to the Hilbert-Kunz multiplicity and $F$-signature of $R / \langle f \rangle$.  Geometrically, one should view this statement as a version of adjunction for the divisor $D = \divisor_{X}(f)$ on $X = \Spec(R)$.


\begin{theorem}
\label{thm.identifyderivatives}
  If $(R, \bm, k)$ is an $F$-finite regular $d$-dimensional local domain and $f \in R$ is a non-zero element, then
\[
D_{-}s(R,f^{1}) = -s(R/ \langle f \rangle) \quad \mbox{ and } \quad D_{+}s(R,f^{0}) = -e_{HK}(R/\langle f \rangle)  .
\]
In other words,  the left derivative of $s(R, f^{t})$ exists at $t = 1$ and equals the negative of the $F$-signature of $R/\langle f \rangle$.  Similarly, the right derivative of $s(R, f^{t})$ exists at $t = 0$ and equals the negative of the Hilbert-Kunz multiplicity of $R/\langle f \rangle$.
\end{theorem}
\begin{proof}
The asserted existence follow immediately from \autoref{cor.DerivativesAt0AreLipschitzConstants}.  For the first statement, we observe by \autoref{prop.EasyComputationOfFSignatureForRegular} that
\[
\begin{array}{rcl}
D_{-}s(R,f^{1}) &=& \displaystyle\lim_{t \to 1^-} \left(\frac{s(R,f^{t})}{t - 1}\right)  = - \displaystyle\lim_{e \to \infty} \left(\frac{s(R,f^{1-1/p^e}) }{1/p^e}\right)\vspace{6pt}\\
 & = & - \displaystyle\lim_{e \to \infty} \left(\frac{{1 \over p^{ed}} \length_R\left(\frac{R}{(\bm^{[p^e]} : f^{p^{e}-1})}\right) }{1/p^e}\right)
    \vspace{6pt} \\ &=&  -\displaystyle\lim_{e \to \infty} {1 \over p^{e(d-1)}} \length_R\left(\frac{R}{\bm^{[p^e]} : (f^{p^{e}}:f)}\right)
\end{array}
\]
which, after switching the sign, equals the $F$-signature of $R/\langle f \rangle$ according to \cite[Corollary 4.14]{BlickleSchwedeTuckerFSigPairs1}.
Next,
we again use \autoref{prop.EasyComputationOfFSignatureForRegular} to observe that
\[
\begin{array}{rcl}
D_{+}s(R,f^{0}) &=& \displaystyle\lim_{t \to 0^+} \left(\frac{s(R,f^{t}) - 1}{t}\right)  = \displaystyle\lim_{e \to \infty} \left(\frac{s(R,f^{1/p^e}) - 1}{1/p^e}\right)\vspace{6pt}\\
 & = & \displaystyle\lim_{e \to \infty} \left(\frac{{1 \over p^{ed}} \length_R\left(R/(\bm^{[p^e]} : f)\right) - 1}{1/p^e}\right)\\
   & = & \displaystyle\lim_{e \to \infty} {1 \over p^{e(d-1)}}\left({\length_R\left(R/(\bm^{[p^e]} : f)\right) - p^{ed} }\right).
\end{array}
\]
However, using \eqref{eq:standardtrick}
and the fact that $R$ is regular, we see that $\length_R\left(R/(\bm^{[p^e]} : f)\right)$ is equal to $p^{ed} - \length_R\left(R/(\bm^{[p^e]} + \langle f \rangle) \right)$.  Therefore:
\[
D_{+}s(R,f^{0})  =  \displaystyle\lim_{e \to \infty}{-1 \over p^{e(d-1)}} \length_R\left(R/(\bm^{[p^e]} + \langle f \rangle) \right)
 =  - e_{HK}(R/\langle f \rangle)
\]
proving the second statement.
\end{proof}

\begin{corollary}
  Suppose that $(R, \bm, k)$ is an $F$-finite regular $d$-dimensional local domain and $f \in R$ is a non-zero element. Then $e_{HK}(R/ \langle f \rangle)$ is the sharp Lipschitz constant for the function $s(R, f^{t})$, and $R/ \langle f \rangle$ is strongly $F$-regular if and only if $s(R, f^{t})$ is \emph{not} differentiable at $t = 1$.
\end{corollary}
\begin{proof}
The first statement follows immediately from \autoref{thm.identifyderivatives} and \autoref{cor.DerivativesAt0AreLipschitzConstants}.  For the second statement, note that $D_{+}s(R,f^{1}) = 0$ as $s(R,f^{t}) = 0$ for all $t \geq 1$.  Thus, $s(R,f^{t})$ is not differentiable at $t = 1$ if and only if $D_{-}s(R,f^{1}) = -s(R/\langle f \rangle)$ is non-zero, which is equivalent to the strong $F$-regularity of $R/ \langle f \rangle$ by \cite{AberbachLeuschke}.
\end{proof}
\begin{figure}[t]
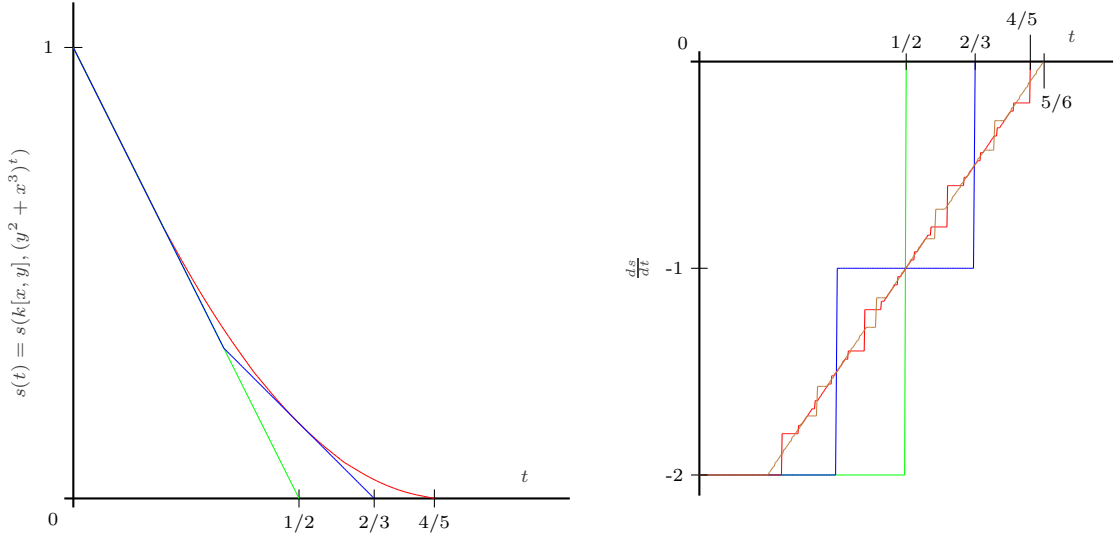


\begin{multicols}{2}


\end{multicols}
\caption{Graphs of the signature function of $s(k[x,y],(y^2+x^3)^t)$ (left) and an approximation of the derivative ${d s \over d t}$ (right) in characteristics {\color{green}2}, {\color{blue}3}, and {\color{red}5}. On the right we also included characteristic  {\color{brown}7} which is at this level of detail virtually indistinguishable from characteristic {\color{red}5} on the left, hence it is omitted there.}
\label{fig.SigCusp}
\end{figure}

We conclude this section with some numerical examples of the function $t \mapsto s(R, f^{t})$ generated in the computer algebra system Macaulay2.  For example, the code
\vskip 6pt
\noindent {\tt{R := ZZ/5[x,y,Degrees=>\{2,3\}]\\for i from 0 to 125 list \\
$\text{\;\;\;\;}$ toString(1-(1/5\textasciicircum 6)*degree(R\textasciicircum 1/(ideal(x\textasciicircum 125,y\textasciicircum 125,(x\textasciicircum 3+y\textasciicircum 2)\textasciicircum i))))}}
\vskip 6pt
\noindent will create the list of the $F$-signatures of $s(\bF_5[x,y], (x^3+y^2)^{i/125})$ for $i = 0, \dots, 125$.
In the graphs shown in \autoref{fig.SigCusp} we plot these values for various characteristics\footnote{See \cite[6.4 Example 4]{MonskyTeixeiraPFractals1} for an alternative analysis of the cusp; to our knowledge, this example is not well-established and understood -- particularly as the characteristic varies.}.  We also plot an approximation of the derivative ${ds \over d t}$  defined by:
\[
{d s \over d t}(i/p^e) \sim {s(R, f^{(i+1)/p^e}) - s(R, f^{i/p^e}) \over {1/p^e}}.
\]
with domain $\{ i / p^e \in \bQ \,|\, 0 \leq i/p^e < 1 \}$.
This approximation of the derivative of the signature function exhibits the same fractal-like behavior observed before in \cite{MonskyTeixeiraPFractals1}.


\begin{figure}[h]
\begin{center}
\begin{tikzpicture}[xscale=100,yscale=1000]
 \tikzstyle{every node}=[font=\tiny]
\begin{scope}[thick]
\draw(0.77,0.0066)--(0.77,-0.00015);
\draw(0.767,0)--(0.84,0);

\draw(4/5,-0.0003) node{4/5};
\draw(9/11,-0.0003) node{9/11};
\draw(5/6,-0.0003) node{5/6};
\draw(0.77,-0.0003) node{0.77};
\draw(0.763,0.006) node{$\sim$0.006};
\draw(0.765,0) node{0};


\begin{scope}[darkgray]
\draw(0.765, 0.0033) node[rotate=90]{$s(t)=s(k[x,y], (y^2 + x^3)^{t})$};

\draw(0.835, 0.0005) node{$t$};
\end{scope}
\end{scope}
\draw(4/5,-0.0001)--(4/5,0.0001);
\draw(9/11,0.0001)--(9/11,-0.0001);
\draw(5/6,0.0001)--(5/6,-0.0001);
\draw(0.769,0.006)--(0.771,0.006);

\begin{scope}[red]
\draw(.7696,.00608)--(.7712,.00576);
\draw(.7712,.00576)--(.7728,.00544);
\draw(.7728,.00544)--(.7744,.00512);
\draw(.7744,.00512)--(.776,.0048);
\draw(.776,.0048)--(.7776,.00448);
\draw(.7776,.00448)--(.7792,.00416);
\draw(.7792,.00416)--(.7808,.00384);
\draw(.7808,.00384)--(.7824,.00352);
\draw(.7824,.00352)--(.784,.0032);
\draw(.784,.0032)--(.7856,.00288);
\draw(.7856,.00288)--(.7872,.00256);
\draw(.7872,.00256)--(.7888,.00224);
\draw(.7888,.00224)--(.7904,.00192);
\draw(.7904,.00192)--(.792,.0016);
\draw(.792,.0016)--(.7936,.00128);
\draw(.7936,.00128)--(.7952,.00096);
\draw(.7952,.00096)--(.7968,.00064);
\draw(.7968,.00064)--(.7984,.00032);
\draw(.7984,.00032)--(.8,0);
\end{scope}

\begin{scope}[brown]
\draw(.769679,.0060774)--(.772595,.00553341);
\draw(.772595,.00553341)--(.77551,.00499792);
\draw(.77551,.00499792)--(.778426,.00452193);
\draw(.778426,.00452193)--(.781341,.00405443);
\draw(.781341,.00405443)--(.784257,.00361244);
\draw(.784257,.00361244)--(.787172,.00319595);
\draw(.787172,.00319595)--(.790087,.00280495);
\draw(.790087,.00280495)--(.793003,.00243946);
\draw(.793003,.00243946)--(.795918,.00208247);
\draw(.795918,.00208247)--(.798834,.00178497);
\draw(.798834,.00178497)--(.801749,.00149598);
\draw(.801749,.00149598)--(.804665,.00123248);
\draw(.804665,.00123248)--(.80758,.000994484);
\draw(.80758,.000994484)--(.810496,.000781987);
\draw(.810496,.000781987)--(.813411,.00059499);
\draw(.813411,.00059499)--(.816327,.000416493);
\draw(.816327,.000416493)--(.819242,.000297495);
\draw(.819242,.000297495)--(.822157,.000186997);
\draw(.822157,.000186997)--(.825073,.000101998);
\draw(.825073,.000101998)--(.827988,.0000424993);
\draw(.827988,.0000424993)--(.830904,.00000849986);
\draw(.830904,.00000849986)--(.8333333,0);
\end{scope}

\begin{scope}[RoyalPurple]
\draw(.769346,.00614091)--(.770098,.0059981);
\draw(.770098,.0059981)--(.770849,.00585642);
\draw(.770849,.00585642)--(.7716,.00571643);
\draw(.7716,.00571643)--(.772352,.00557813);
\draw(.772352,.00557813)--(.773103,.00544153);
\draw(.773103,.00544153)--(.773854,.00530662);
\draw(.773854,.00530662)--(.774606,.0051734);
\draw(.774606,.0051734)--(.775357,.00504188);
\draw(.775357,.00504188)--(.776108,.00491149);
\draw(.776108,.00491149)--(.77686,.00478109);
\draw(.77686,.00478109)--(.777611,.00465691);
\draw(.777611,.00465691)--(.778362,.00453273);
\draw(.778362,.00453273)--(.779113,.00440967);
\draw(.779113,.00440967)--(.779865,.00428831);
\draw(.779865,.00428831)--(.780616,.00416864);
\draw(.780616,.00416864)--(.781367,.00405066);
\draw(.781367,.00405066)--(.782119,.00393438);
\draw(.782119,.00393438)--(.78287,.0038198);
\draw(.78287,.0038198)--(.783621,.0037069);
\draw(.783621,.0037069)--(.784373,.00359513);
\draw(.784373,.00359513)--(.785124,.00348337);
\draw(.785124,.00348337)--(.785875,.00337781);
\draw(.785875,.00337781)--(.786627,.00327226);
\draw(.786627,.00327226)--(.787378,.00316783);
\draw(.787378,.00316783)--(.788129,.00306509);
\draw(.788129,.00306509)--(.788881,.00296405);
\draw(.788881,.00296405)--(.789632,.00286471);
\draw(.789632,.00286471)--(.790383,.00276705);
\draw(.790383,.00276705)--(.791134,.00267109);
\draw(.791134,.00267109)--(.791886,.00257682);
\draw(.791886,.00257682)--(.792637,.00248369);
\draw(.792637,.00248369)--(.793388,.00239055);
\draw(.793388,.00239055)--(.79414,.00230362);
\draw(.79414,.00230362)--(.794891,.00221669);
\draw(.794891,.00221669)--(.795642,.00213089);
\draw(.795642,.00213089)--(.796394,.00204678);
\draw(.796394,.00204678)--(.797145,.00196437);
\draw(.797145,.00196437)--(.797896,.00188365);
\draw(.797896,.00188365)--(.798648,.00180462);
\draw(.798648,.00180462)--(.799399,.00172729);
\draw(.799399,.00172729)--(.80015,.00165165);
\draw(.80015,.00165165)--(.800902,.00157714);
\draw(.800902,.00157714)--(.801653,.00150263);
\draw(.801653,.00150263)--(.802404,.00143433);
\draw(.802404,.00143433)--(.803156,.00136603);
\draw(.803156,.00136603)--(.803907,.00129773);
\draw(.803907,.00129773)--(.804658,.00122942);
\draw(.804658,.00122942)--(.805409,.00116112);
\draw(.805409,.00116112)--(.806161,.00109282);
\draw(.806161,.00109282)--(.806912,.00102452);
\draw(.806912,.00102452)--(.807663,.000956219);
\draw(.807663,.000956219)--(.808415,.000887917);
\draw(.808415,.000887917)--(.809166,.000819616);
\draw(.809166,.000819616)--(.809917,.000751315);
\draw(.809917,.000751315)--(.810669,.000683013);
\draw(.810669,.000683013)--(.81142,.000614712);
\draw(.81142,.000614712)--(.812171,.000546411);
\draw(.812171,.000546411)--(.812923,.000478109);
\draw(.812923,.000478109)--(.813674,.000409808);
\draw(.813674,.000409808)--(.814425,.000341507);
\draw(.814425,.000341507)--(.815177,.000273205);
\draw(.815177,.000273205)--(.815928,.000204904);
\draw(.815928,.000204904)--(.816679,.000136603);
\draw(.816679,.000136603)--(.817431,.0000683013);
\draw(.817431,.0000683013)--(.818182,0);
\end{scope}
\end{tikzpicture}
\end{center}
\caption{Graph of the $F$-signature function $s(k[x,y],(x^2+y^3)^t)$ near the $F$-pure threshold in characteristics {\color{red}5}, {\color{brown}7}, and {\color{RoyalPurple}11}.  Away from the $F$-pure threshold they are harder to distinguish at this level of detail.}
\end{figure}
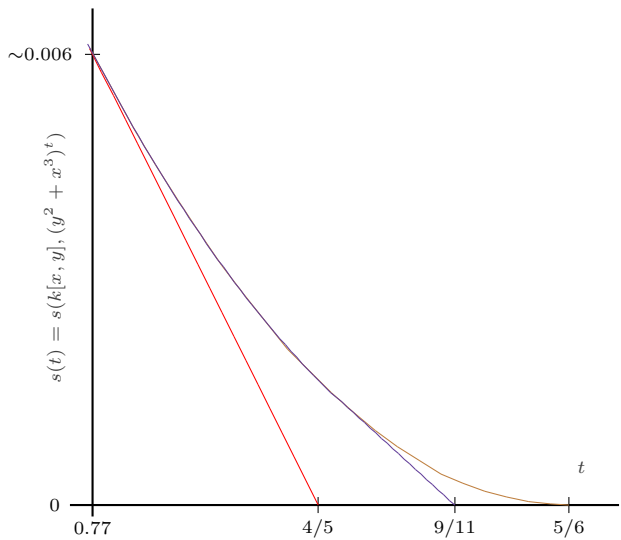

In these examples, the signature function is both continuous and convex as predicted by the results of \autoref{sec.Continuity}.  However, further inspection raises a number of natural questions.

\begin{question}
Set $R = \bF_p[x_1, \dots, x_n]_{\langle x_1, \dots, x_n \rangle}$ and fix $f \in R$.  Does the function $t \mapsto s(R, f^{t})$ have any sort of uniform behavior as $p$ goes to infinity?  If so, does the limit of these have a geometric interpretation in terms of the log resolution of $(R,f)$?
\end{question}

\begin{question}
Set $R = \bF_p[x_1, \dots, x_n]_{\langle x_1, \dots, x_n \rangle}$ and fix $f \in R$.  Does the function $t \mapsto s(R, f^{t})$ send rational numbers to rational numbers?
\end{question}

Based upon some further computer experimentation, we also ask.

\begin{question}
Set $R = \bF_p[x_1, \dots, x_n]_{\langle x_1, \dots, x_n \rangle}$ and fix $f \in R$.  When is the function $t \mapsto s(R, f^{t})$ piece-wise polynomial?  If so, what happens when $p$ goes to infinity for these pieces?
\end{question}

\begin{remark}
The work and computations of Monsky and Teixeira -- in light of \eqref{eq:pfractalrelationship} -- directly gives some substantial examples of the function $t \mapsto s(R, f^{t})$.  For example,  it follows from \cite[6.1 Example 1]{MonskyTeixeiraPFractals1} that for $t \in [0,1]$ we have
\[
s(\bF_{3}\llbracket x,y \rrbracket, (y^{3}-x^{4} + x^{2}y^{2})^{t}) = \frac{9 - 36t + 36t^{2}-\Delta^{2}(t)}{8}
\]
where $\Delta \: [0,1] \to \bR$ is the fractal-like function pictured in \cite[Figure 1]{MonskyTeixeiraPFractals1} determined by setting $\Delta(t) = 6t-3$ for $t \in [2/3, 1]$ together with the relations $\Delta(\frac{t}{3}) = \frac{1}{3}\Delta(1-t)$ and $\Delta(\frac{t+1}{3}) = \frac{1}{3}\Delta(t)$.
\end{remark}

\begin{remark}
As visible in the examples above, the functions $t \mapsto s(R,f^{t})$ exhibit rather complex and interesting behavior.  However, very few broad classes of examples have been computed: for example, little is known even in the case of a homogenous hypersurfaces with an isolated singularity.  In light of Theorem~\ref{thm.identifyderivatives}, computing these functions is a problem at least as difficult as explicitly computing Hilbert-Kunz multiplicities --  widely considered an extremely challenging problem.

Further evidence for the assertion that computing Hilbert-Kunz multiplicities is extremely difficult can be seen from \cite{MonskyRationalityofHKCounterexampleConjecture}.  Here, Monsky considers the homogenous cubic hypersurface defined by $x^{3}+y^{3}+xyz$ in $\bF_{2} \llbracket x,y,z \rrbracket$.  While the Hilbert-Kunz multiplicity of this hypersurface is known by \cite{BuchweitzChenHKfunctionsofCubics}, the corresponding function $t \mapsto s(\bF_{2}\llbracket x,y \rrbracket, (x^{3}+y^{3}+xyz)^{t})$ remains elusive.  In fact, again using \eqref{eq:pfractalrelationship}, \cite[Conjecture 1.5]{MonskyRationalityofHKCounterexampleConjecture} essentially conjectures a recursive relationship that completely determines this function (\textit{cf. loc. cit.} Theorem 1.9).  The importance of this conjecture stems from the fact that, if true, it implies that the weighted homogeneous hypersurface $uv + x^{3}+y^{3}+xyz$ in $\bF_{2}\llbracket x,y,z,u,v \rrbracket$ has irrational Hilbert-Kunz multiplicity and $F$-signature; see \cite[Corollary 2.7]{MonskyRationalityofHKCounterexampleConjecture} and \cite[Proposition 4.22]{TuckerFSignatureExists}.

\end{remark}

\section{Semi-continuity of the \texorpdfstring{$F$}{F}-signature of a pair}


Suppose that $R$ is a \emph{regular} $F$-finite domain and fix an element of $f \in R$ and positive integers $a, c$. We are interested in the signature $s(R_{\bq},f^{a/p^c})$ as a function in $\bq \in \Spec R$. Due to the special shape of the exponent $a/p^c$ we may use \autoref{prop.EasyComputationOfFSignatureForRegular} to see that this function
\[
    H \colonequals \Spec R \to \bR; \quad \bq \mapsto s(R_{\bq},f^{a/p^c}) = { 1 \over p^{c(\dim R_{\bq})}} \ell_{R_{\bq}}\big(R_{\bq}/ (\bq^{[p^c]} : f^{a})\big)
\]
is given by a single length. It is precisely this observation that enables us to show that $H$, and in consequence the $F$-signature for exponents of a special shape, is lower semi-continuous. Note that in general, where the $F$-signature is not known to be computed by a single length, it is much harder to obtain a semi-continuity result. The lower semi-continuity of the $F$-signature for arbitrary exponents will be derived afterwards with the help of our \emph{continuity in $t$}  result \autoref{thm.Continuity}.
\begin{proposition}
\label{prop.BaseSemiContinuityAsPointVaries}
Suppose that $R$ is a regular $F$-finite ring.  Suppose additionally that $f \in R$ and also $a, c \in \bN$.  Then the function
\[
H : \bq \in \Spec R \mapsto s(R_{\bq}, f^{a/p^c})
\]
is lower semi-continuous on $\Spec R$.
\end{proposition}
\begin{proof}
As explained above, due to \autoref{prop.EasyComputationOfFSignatureForRegular}, we have to show that the expression ${1 \over p^{c \dim R_{\bq}}}\length_{R_{\bq}}(R_{\bq}/(\bq^{[p^c]} : f^a))$ is lower semi-continuous.  Using \eqref{eq:standardtrick}, we have
\[
{1 \over p^{c \dim R_{\bq}}}\length_{R_{\bq}}\left(R_{\bq}/(\bq^{[p^c]} : f^a)\right) = 1 - {1 \over p^{c \dim R_{\bq}}}\length_{R_{\bq}}\left(R_{\bq}/\langle \bq^{[p^c]} ,f^a \rangle \right)
\]
and so it is equivalent to show that the function
\[
\bq \in \Spec R \mapsto  {1 \over p^{c \dim R_{\bq}}}\length_{R_{\bq}}\left(R_{\bq}/\langle \bq^{[p^c]} ,f^a \rangle \right)
\]
is upper semi-continuous.  Since $R$ is excellent as it is $F$-finite \cite{KunzOnNoetherianRingsOfCharP} and $R/ \langle f^{a} \rangle$ is locally equidimensional (it is Cohen-Macaulay, and thus locally unmixed), this follows immediately from \cite{ShepherdBarronOnAProblemOfKunz}.  However, we include herein an alternative proof in our current setting.

Tensoring the short exact sequence
\[
    0 \to R_{\bq} \cramped{\to[\cdot f^a]} R_{\bq} \to R_{\bq}/\langle f^a \rangle \to 0
\]
 with $R_{\bq}/\bq^{[p^c]}$, we may extract from the resulting $\Tor$-sequence the equality
\begin{equation}
\label{eq:wasstar}
 {1 \over p^{c \dim R_{\bq}}}\length_{R_{\bq}}\left(\Tor_1^{R_{\bq}}(R_{\bq}/f^a, R_{\bq}/\bq^{[p^c]})\right) =  {1 \over p^{c \dim R_{\bq}}}\length_{R_{\bq}}\left( R_{\bq}/\langle \bq^{[p^c]} ,f^a \rangle  \right).
\end{equation}
Since for any finite length module $M$ one has $\length_{R_\bq}(F^c_*M)=p^{\alpha(R_\bq)} \length_{R_\bq}(M)$ we get that \eqref{eq:wasstar} is equal to
\[
     \frac{1}{p^{c(\alpha(R_\bq)+\dim R_{\bq})}} \length_{R_\bq}\left(F^c_*\Tor_1^{R_{\bq}}(R_{\bq}/\langle f^a \rangle, R_{\bq}/\bq^{[p^c]})\right)
\]
By \cite[Proposition 2.3]{KunzOnNoetherianRingsOfCharP} one knows that $\alpha(R_{\bq}) + \dim R_{\bq}$ is constant on $\Spec R$, hence so is the coefficient above. Using that $F^c_*(R_\bq/\bq^{[p^c]}) \cong F^c_*R_{\bq} \tensor k(\bq)$ we also have
\begin{align*}
F^c_* \Tor_1^{R_{\bq}}(R_{\bq}/\langle f^a \rangle , R_{\bq}/\bq^{[p^c]})
& \cong \Tor_1^{F^c_* R_{\bq}}(F^c_* (R_{\bq}/\langle f^a \rangle), (F^c_* R_{\bq}) \tensor_{R} k({\bq}) ) \\
& \cong \Tor_1^{R}(F^c_* (R/\langle f^a \rangle), k(\bq))
\end{align*}
where the second isomorphism is due to base-change for Tor, \cite[Proposition 3.2.9]{WeibelHomological}, which we can apply because $F^c_* R_{\bq}$ is a flat $R_{\bq}$-module (since $R$ is regular \cite{KunzCharacterizationsOfRegularLocalRings}). Now, the length of $\Tor_1^{R}(F^c_* (R/\langle f^a \rangle), k(\bq))$ (as an $R_{\bq}$-module) is upper semi-continuous in $\bq \in \Spec R$ by \cite[Theor\`eme 7.6.9]{EGAIII2}. This completes our proof.
\end{proof}
Combining the previous result with the continuity of $F$-signature of \autoref{thm.Continuity}, we obtain the following improvement.
\begin{corollary}\label{thm.SemiContinuityPair}
Suppose that $R$ is a regular $F$-finite ring.  Suppose additionally that $f \in R$ and $t > 0$ is a positive real number.  Then the function
\[
\bq \in \Spec R \mapsto s(R_{\bq}, f^t)
\]
is lower semi-continuous on $\Spec R$.
\end{corollary}
\begin{proof}
Fix $\varepsilon > 0$ and choose $\bq_0 \in \Spec R$.  By the continuity of the function $\lambda \mapsto s(R_{\bq}, f^\lambda)$, we know that there exists a $\delta > 0$ such that $s(R_{\bq}, f^{\lambda}) \in \big( s(R_{\bq}, f^{\lambda}) - \varepsilon/2, s(R_{\bq}, f^{\lambda}) + \varepsilon/2\big)$ for all $\lambda \in (t - \delta, t + \delta)$.  Choose positive integers $a$ and $c$, such that $t + \delta > a/p^c \geq t$.  Thus
\[
s(R_{\bq}, f^t) \geq s(R_{\bq}, f^{a/p^c}) > s(R_{\bq}, f^t) - \varepsilon/2.
\]
Using \autoref{prop.BaseSemiContinuityAsPointVaries}, choose an open neighborhood $U$ of $\bq$ such that $s(R_{\bp}, f^{a/p^c}) > s(R_{\bq}, f^{a/p^c}) - \varepsilon/2$ for all $\bp \in U$.  But then
\[
s(R_{\bp}, f^t) \geq s(R_{\bp}, f^{a/p^c}) > s(R_{\bq}, f^{a/p^c}) - \varepsilon/2 > s(R_{\bq}, f^t) - \varepsilon
\]
for all $\bp \in U$, proving that $s(R_{\bp}, f^t)$ is lower semi-continuous as a function of $\bp \in \Spec R$ as desired.
\end{proof}

\section{Minimal log discrepancies}
\label{sec:minim-log-discr}

We now return to the setting of arbitrary triples $(R, \Delta, \ba^{t})$ (\textit{i.e.} where $R$ is not assumed regular and $\ba$ not assumed principal).  In general, while the test ideal $\tau(R, \Delta, \ba^t)$ measures the singularities of a given triple, it only provides information if the triple is \emph{not} strongly \mbox{$F$-regular}.  Complementarily, the $F$-signature $s(R, \Delta, \ba^{t})$ provides information on singularities which are strongly \mbox{$F$-regular}.  In characteristic zero, the analog of the test ideal is the multiplier ideal $\mJ(R, \Delta, \ba^t)$, which only provides information on non-Kawamata log terminal singularities; the most common tool for measuring singularities which are Kawamata log terminal is the \emph{minimal log discrepancy}.  This leads one to hope for a connection between the \mbox{$F$-signature} and the minimal log discrepancy.  In this section, we show that the $F$-signature of a triple is always a lower bound for the minimal log discrepancy.  We should point out that we expect that this relationship can be substantially improved.

Throughout this section we assume that $R$ is a normal and $F$-finite ring. We begin by recalling the definition of the minimal log discrepancy.

\begin{definition}[\protect{\cf~\cite[Section 2.3]{KollarMori}}]
Suppose that $(R, \Delta, \ba^t)$ is a triple with $X = \Spec R$ and $K_X + \Delta$ assumed to be $\bQ$-Cartier.  For any prime divisor $E$ (possibly non-exceptional) on a $\tld X$, a normal variety with a proper birational map $\pi : \tld X \to X = \Spec R$ such that $\ba \cdot \cO_{\tld X} = \cO_{X}(-G)$, we define the \emph{discrepancy of $(R, \Delta, \ba^t)$ along $E$}, denoted $b_E$ to be the coefficient of $E$ in $K_{\tld X} - \pi^*(K_X + \Delta) - t G$.  The \emph{log discrepancy of $(R, \Delta, \ba^t)$ along $E$} is simply the discrepancy plus one, $b_E + 1$.

Given a (possibly non-closed) point $x \in X = \Spec R$, the \emph{minimal log discrepancy of $(R, \Delta, \ba^t)$ at $x$} is simply
\begin{multline*}
\inf \{ b_E + 1 \, |\, \pi : \tld X \to X \text{ proper birational with $\tld X$ normal, $\ba \cdot \O_{\tld X}$ locally free,} \\
\text{and }  E \subseteq \tld X \text{ a prime divisor with }\pi(E) = x\}
\end{multline*}
possibly equalling negative infinity.  We denote this number by $\mld(x; X, \Delta, \ba^t)$.
\end{definition}

\begin{remark}
It is easy to see that $(R, \Delta, \ba^t)$ is Kawamata log terminal if and only if we have $\mld(x; X, \Delta, \ba^t) > 0$ for every point $x \in X = \Spec R$ (including codimension 1 points).
\end{remark}

We first prove a transformation rule for $F$-splitting numbers under birational maps, the main idea for the proof goes back to at least \cite{HaraWatanabeFRegFPure} and \cite{MehtaSrinivasFPureSurface}.

\begin{proposition}
\label{prop.TransformationOfSplittingNumbersUnderBirational}
Suppose that $(R, \Delta, \ba^t)$ is a triple where $X = \Spec R$ with $K_{X} + \Delta$ is $\Q$-Cartier, and $\pi : \tld X \to X = \Spec R$ is a proper birational map with $\tld X$ normal and $\ba \cdot \cO_X(-G)$.  Further suppose that $\eta \in \tld X$ is a possibly (non-closed) point mapping to a (possibly non-closed) point $x \in X$ and set $\tld \Delta = -K_{\tld X} + \pi^*(K_X + \Delta)$.  Then the \mbox{$F$-splitting} numbers $a_{e,x}^{\Delta, \ba^t}$ of the local ring $\mathcal{O}_{X,x}$ are less than or equal to the \mbox{$F$-splitting} numbers $a_{e,\eta}^{\tld \Delta + tG}$ of $\mathcal{O}_{\tld{X}, \eta}$, assuming $ \tld \Delta + tG$ is effective at $\eta$.
\end{proposition}
\begin{proof}
Suppose that $(\oplus_i \varphi_i) : F^e_* \O_{X,x} \to \oplus_i \O_{X,x}$ is a finite direct sum of $(R, \Delta, \ba^t)$-summands and thus notice that $(\oplus_i \varphi_i)$ is surjective.  We claim that each $\varphi_i$ induces a map
\[
\tld{\varphi}_i : F^e_* \O_{\tld X, \eta} \subseteq F^e_* \O_{\tld X, \eta}(\lceil (p^e-1)(\tld \Delta + tG)\rceil) \to \O_{\tld X, \eta}
\]
which agrees with $\varphi_i$ on the field of fractions $K(X) = K(\tld X)$.  This is well known to experts, see \cite[Main Theorem]{HaraWatanabeFRegFPure} and \cf~\cite[ Theorem 6.7]{SchwedeCentersOfFPurity}, but we briefly sketch the idea.   Any $\varphi_i$ can be written as $\varphi_i = \sum_j a_{ij} \varphi_{ij}$ with $\varphi_{ij} \in \Hom_R(R(\lceil (p^e-1) \Delta \rceil), R)$ and $a_{ij} \in \ba^{\lceil t(p^e - 1) \rceil}$.  Each $\varphi_{ij}$ corresponds to a divisor $\Delta_{ij} \geq \Delta$ as in \cite[Section 4.4]{SchwedeTuckerTestIdealSurvey}.  By extending $\varphi_i$ and the $\varphi_{ij}$ to the fraction field $K(X)$ of $X$, we obtain $\varphi_i, \varphi_{ij} : F^e_* K(X) \to K(X)$.  We then restrict the domains of these maps to the local ring $F^e_* \O_{\tld X, \eta}$; we denote these restrictions by $\tld \varphi$.  The $\tld \varphi_{ij}$, at $\eta$, correspond to the divisors $-K_{\tld X} + \pi^*(K_X + \Delta_{ij}) \geq -K_{\tld X} + \pi^*(K_X + \Delta) = \tld \Delta$ (see the aforementioned references).  Then $\tld \varphi = \sum_j a_{ij} \tld \varphi_{ij}$ and it follows that the divisor corresponding to $\tld \varphi$, at least at $\eta$, is bigger than or equal to $\tld \Delta + tG$ as desired.

Regardless, taking the direct sum of the $\tld{\varphi}_i$, we obtain a map
\[
(\oplus_i \tld \varphi_i) : F^e_* \O_{\tld X, \eta} \subseteq F^e_* \O_{\tld X, \eta}(\lceil (p^e-1)(\tld \Delta + tG) \rceil) \to \oplus_i \O_{\tld X, \eta}.
\]
Each section of $F^e_* \O_{X}$ is also a section of $F^e_* \O_{\tld X, \eta} \subseteq F^e_* \O_{\tld X, \eta}(\lceil (p^e-1)\tld (\Delta + tG) \rceil)$.  Thus if $s \in F^e_* \O_{X}$ is sent to the $j$th basis element $e_j \in \oplus_i \O_{X,x}$ by $\oplus_i \varphi_i$, we also have that $s \in F^e_* \O_{\tld X, \eta} \subseteq F^e_* \O_{\tld X, \eta}(\lceil (p^e-1)(\tld \Delta + tG) \rceil)$ is sent to $e_j \in \oplus_i \O_{\tld X, \eta}$.  But this implies that $\oplus \tld\varphi_i$ is surjective as desired.
\end{proof}

\begin{corollary}\label{thm.F-sigLeqMLD}
Suppose that $(R, \Delta, \ba^t)$ is a strongly $F$-regular triple with $x \in X = \Spec R$.  Then $s(R_x, \Delta, \ba^t) \leq {\mld}(x; X, \Delta, \ba^t)$.
\end{corollary}
\begin{proof}
Suppose $\pi : \tld X \to X$ is a proper birational map with $\tld X$ normal and consider a prime divisor $E \subseteq \tld X$ with generic point $\eta$ such that $\pi(\eta) = x \in X$.  Set $b$ to be the coefficient of $E$ in $-\tld \Delta - tG := K_{\tld X} - \pi^*(K_X + \Delta) -t G$.  If $b > 0$, then the log discrepancy of $(\tld X, \tld \Delta)$ at $\eta$ is positive and there is nothing to prove.  On the other hand, if $b < 0$, then we see $b > -1$ by the strongly $F$-regular assumption and \cite[Main Theorem]{HaraWatanabeFRegFPure}.  Now then, $s(\O_{\tld X, \eta}, \tld \Delta) = 1 + b$ by \autoref{eq.SignatureOfSNC} on page \pageref{eq.SignatureOfSNC}.  By \autoref{prop.TransformationOfSplittingNumbersUnderBirational}, and the fact that $d := \dim(R_x) + \alpha(R_x) = \dim(\O_{\tld X, \eta}) + \alpha(\O_{\tld X, \eta})$ by \cite[Proposition 2.3]{KunzOnNoetherianRingsOfCharP}, we see that
\[
\begin{array}{rl}
s(R, \Delta, \ba^t) = & \lim_{e \to \infty} {a_{e,x}^{\Delta,\ba^t} \over p^{ed}} \vspace{6pt} \\
\leq & \lim_{e \to \infty} {a_{e,\eta}^{\tld \Delta +tG} \over p^{ed}} \vspace{6pt} \\
= & s(\O_{\tld X, \eta}, \tld \Delta) = 1 + b.
\end{array}
\]
The result follows.
\end{proof}

\begin{remark}
The previous proof in fact shows that the log discrepancy of any divisor $E \subseteq \tld X$ such that $x \in \pi(E)$ is $\geq s(X, \Delta, \ba^t)$.
\end{remark}

We can obtain the following improvement as well.

\begin{corollary}
\label{cor.MldImprovement}
Suppose that $(R, \Delta, \ba^t)$ is a strongly $F$-regular pair with $x \in X = \Spec R$.  Further suppose that $\pi : \tld X \to X$ is a proper birational map and $\eta \in \tld X$ is such that $\pi(\eta) = x$.  Finally suppose that $\tld \Delta = -K_{\tld X} + \pi^*(K_X + \Delta)$ is an effective simple normal crossings divisor at $\eta$ with components $E_j$ passing through $\eta$ and with coefficients $b_j$.  Then $s(R, \Delta) \leq \prod (1 - b_j)$.
\end{corollary}
\begin{proof}
This follows from the argument above again using Proposition \ref{prop.TransformationOfSplittingNumbersUnderBirational} and \cite[Example 4.19]{BlickleSchwedeTuckerFSigPairs1} (also see Section \ref{sec.Defn}).
\end{proof}

\begin{remark}
The results in this section are often far from optimal.  In particular, the statements of this section are trivial for any pair with canonical singularities.  As such, it is interesting to ask if any of the bounds in this section can be improved upon.  However, given the limited number of examples of $F$-signature that have been computed, it is unclear to us what form any sharpenings of this result may take on.
\end{remark}

\bibliographystyle{skalpha}
\bibliography{CommonBib}

\end{document}